\documentclass[11pt]{amsart}

\usepackage[a4paper,left=35mm,right=37mm,top=37mm,bottom=40mm,marginpar=20mm]{geometry}

\usepackage[utf8]{inputenc}
\usepackage{mathtools}
\usepackage{amsmath}
\usepackage{amssymb}
\usepackage{amsthm}
\usepackage{amscd}
\usepackage[initials,short-journals,msc-links,nobysame]{amsrefs}
\usepackage{url}
\usepackage{bbm}
\usepackage{color}
\usepackage[english=american]{csquotes}
\usepackage{hyperref}
\usepackage{tikz}
\usepackage{bm}
\usepackage{color}
\usepackage{calc}
\usepackage{enumerate}
\usepackage{mathrsfs}
\usepackage{hyperref}
\hypersetup{
	colorlinks=true,
	citecolor=pink!75!black,
	linkcolor=teal,
	filecolor=black,      
	urlcolor=blue,
}
\usepackage[all]{xy}

\DeclareMathOperator{\id}{id}

\DeclareMathOperator{\im}{im}

\DeclareMathOperator{\curl}{curl}

\DeclareMathOperator{\rank}{rank}

\newcommand{\K}{\mathbb{K}}

\newcommand{\set}[2]{\left\{\, #1 \ \textup{{:}}\ #2 \,\right\}}

\newcommand{\R}{\mathbb{R}}
\newcommand{\C}{\mathbb{C}}
\newcommand{\Q}{\mathbb{Q}}

\newcommand{\sym}{\mathrm{sym}}

\newcommand{\embed}{\hookrightarrow}

\newcommand{\sbullet}{\begin{picture}(1,1)(-0.5,-2)\circle*{2}\end{picture}}
\newcommand{\frarg}{\,\sbullet\,}

\DeclareMathOperator{\Curl}{Curl}

\newcommand{\cB}{\mathcal{B}}

\newtheoremstyle{thmlemcorr}{10pt}{10pt}{\itshape}{}{\itshape\bfseries}{.}{3pt}{{\thmname{#1}\thmnumber{ #2}\thmnote{ (#3)}}}
\newtheoremstyle{thmlemcorr*}{15pt}{15pt}{\itshape}{15pt}{\bfseries}{.}\newline{{\thmname{#1}\thmnumber{ #2}\thmnote{ (#3)}}}
\newtheoremstyle{defi}{15pt}{15pt}{}{15pt}{\itshape\bfseries}{.}{5pt}{{\thmname{#1}\thmnumber{ #2}\thmnote{ (#3)}}}
\newtheoremstyle{remexample}{7pt}{7pt}{}{}{\itshape}{.}{5pt}{{\thmname{#1}\thmnumber{ #2}\thmnote{ (#3)}}}
\newtheoremstyle{ass}{10pt}{10pt}{}{}{\bfseries}{.}{10pt}{{\thmname{#1}\thmnumber{ }\thmnote{ (#3)}}}

\theoremstyle{thmlemcorr}
\newtheorem{theorem}{Theorem}[section]
\newtheorem{lemma}[theorem]{Lemma}
\newtheorem{corollary}[theorem]{Corollary}
\newtheorem{proposition}[theorem]{Proposition}

\theoremstyle{thmlemcorr*}
\newtheorem{theorem*}{Theorem}
\newtheorem{lemma*}[theorem]{Lemma}
\newtheorem{corollary*}[theorem]{Corollary}
\newtheorem{proposition*}[theorem]{Proposition}
\newtheorem{problem*}[theorem]{Problem}
\newtheorem{conjecture*}[theorem]{Conjecture}

\theoremstyle{defi}
\newtheorem{definition}[theorem]{Definition}

\theoremstyle{remexample}
\newtheorem{remark}[theorem]{Remark}
\newtheorem{example}[theorem]{Example}

\theoremstyle{ass}

\newcommand{\cA}{\mathcal{A}}
\newcommand{\cQ}{\mathcal{Q}}

\title[Homological properties of constant rank operators]{An elementary approach to the homological properties of constant-rank operators}
\author{Adolfo Arroyo-Rabasa}
\address{Universit\'e catholique de Louvain, 
Institut de Recherche en Mathématique et Physique,
Chemin du Cyclotron 2,
1348 Louvain-la-Neuve, Belgium}
\email{adolforabasa@gmail.com,adolfo.arroyo@uclouvain.be}
\author{Jos\'e Simental}
\address{Max-Planck Institute for Mathematics. Vivatsgasse 7. Bonn, Germany, 53111}
\email{jose@mpim-bonn.mpg.de}
\begin{document}

\maketitle
\begin{abstract}
  We give a simple and constructive extension of Rai\c{t}\u{a}'s result that every constant-rank operator possesses an exact potential and an exact annihilator. Our construction is completely self-contained and provides an improvement on the order of the operators constructed by Rai\c{t}\u{a}, as well as the order of the explicit annihilators for elliptic operators due to Van Schaftingen. We also give an abstract  construction of an optimal annihilator for constant-rank operators, which extends the optimal construction of Van Schaftingen for elliptic operators. Lastly, we establish a generalized Poincar\'e lemma for constant-rank operators and homogeneous spaces on $\mathbb{R}^d$, and we show that the existence of potentials on spaces of periodic maps requires a strictly weaker condition than the constant-rank property.\\
    
		\vspace{4pt}
		
		\vspace{5pt}
\end{abstract}

\tableofcontents

\section{Introduction}

Let $V$, $W$ be $\mathbb{R}$-vector or $\C$-vector spaces of dimensions $N, M$. 
We consider a homogeneous differential operator on  $\R^d$ 
from $V$ to $W$ with constant (real or complex) coefficients, that is,
\[
    \cA(D) = \sum_{|\alpha| = k} A_\alpha \partial^\alpha, 
\]
where the coefficients $A_\alpha$ belong to $\mathrm{Lin}(V,W)$, $\alpha = (\alpha_1,\dots,\alpha_d) \in \mathbb N_{0}^d$ is a multi-index with modulus $|\alpha| = \alpha_1 + \dots + \alpha_d = k$, and $\partial^\alpha$ is the composition of  partial distributional derivatives $\partial_1^{\alpha_1} \cdots \partial_d^{\alpha_d}$. As our main and only assumption, we require that $\cA(D)$ satisfies the constant-rank property: there exists a non-negative integer $r$ such that
\begin{equation}\label{eq:rank}
    \rank A(\xi) = r \quad \text{for all $\xi \in \R^d - \{0\}$,}
\end{equation}
where 
\[
     A(\xi) \coloneqq \sum_{|\alpha| = k} A_\alpha\xi^{\alpha}, \qquad \xi^\alpha \coloneqq \xi_1^{\alpha_1} \cdots \xi_d^{\alpha_d}, \qquad \xi \in \R^d,
\]
is the principal symbol associated to the operator $\cA(D)$. The symbol $A(\xi)$ is precisely the  coefficient representation of $\cA(D)$ in Fourier space, that is,
\[
    (\cA f)\;\widehat{}\;(\xi) = 
(2\pi i)^k A(\xi) \widehat f(\xi)
\]
for all Schwartz maps $f \in \mathcal S(\R^d;V)$. Our setting thus covers all linear homogeneous systems of real and complex constant coefficients acting on maps over $\R^d$, although many of our results will be stated in more generality for symbols over   $\mathbb K^d$, where $\K$ is any field.

In order to motivate this framework, let us briefly discuss its origins as well as some elements of its more recent theory. Operators of constant rank were considered by Schulenberger \& Wilcox~\cite{SW1} to prove Hilbert-space coercive inequalities
\[
    \|Du\|_{L^2} \le C \left(\|\cA u\|_{L^2} + \|u\|_{L^2}\right)\,,
\]
for non-elliptic first-order operators in full space (see also~\cite{kato,sarason,wilcox}). 
In~\cite{Mur81}, Murat built upon these ideas to  establish that~\eqref{eq:rank} is a \emph{sufficient} condition for the $L^p$-boundedness of the (extension of the)  canonical $L^2$-projection $P : C^\infty_c(\R^d) \to C^\infty_c(\R^d)$ onto $\ker \cA(D)$, which he also showed satisfied 
\[
\|u - Pu\|_{L^p} \le C(p,A) \|u\|_{L^p}
\]
for all $1 < p < \infty$. Murat's work nourished the development of the \emph{compensated compactness theory} for Sobolev spaces associated with anisotropic operators (see~\cite{Mur81} and references therein). These inequalities would later be improved by Fonseca and M\"uller~\cite{FM99} (see also~\cite{ADR18} where the trivial extension to higher order operators is established), who demonstrated that Murat's $L^p$-projection for constant-rank operators gives rise to a Korn-type estimate
\begin{equation}\label{eq:FMP}
    \|D^k(u - Pu)\|_{L^p} \le  C(p,A)\|\mathcal A u\|_{L^p}\,.
\end{equation}
Recently, Guerra and Raiț{\u{a}}~\cite{GR} showed that the constant-rank property is also a \emph{necessary} condition for the validity of~\eqref{eq:FMP}. 

Lastly, and crucial to the motivation for the content of this note, Raiț{\u{a}} proved in~\cite[Theorem~1]{raita2018potentials} that the constant-rank property~\eqref{eq:rank} is also a sufficient and necessary condition for the existence of potentials associated with real-coefficients constant-rank operators. More precisely, he proved that a real-coefficient operator $\mathcal A(D)$ has constant rank if and only if there exists a homogeneous polynomial $B: \R^d \to \mathrm{Lin}(V,V)$ such that
\begin{equation}\label{eq:exact}
\im B(\xi) = \ker A(\xi) \quad \text{for all non-zero $\xi \in \R^d$.}
\end{equation}

He exploited this purely algebraic \emph{homological} property to show that, when restricted to sufficiently regular mean-value zero $\mathbb Z^d$-periodic maps $v : \R^d / \mathbb Z^d \to V$, the constant rank assumption implies (but is not equivalent) with the following homological property: 
\[
    \mathcal Av = 0 \quad \Longrightarrow \quad  v  = \mathcal Bu \quad \text{for some $u : \R^d / \mathbb Z^d \to W$.}
\]
This homology-type result has proved to be a very useful tool to solve some longstanding questions in the calculus of variations related to the study of oscillations and concentration effects associated to sequences of PDE-constrained maps (see~\cite{adolfo,gr1,gr2,kristensen}).

\subsection{Summary of the main results} The first goal of this note is to give an alternative, rather elementary and self-contained generalization  of~\eqref{eq:exact} for symbol maps over arbitrary fields $\K$, which avoids the a computation via the Moore--Penrose pseudo-inverse of the principal symbol map $\xi \mapsto A(\xi)$. Our construction (see Lemma~\ref{thm:Q} and Theorem~\ref{thm:poincare}) is a potential $B: \K^d \to \mathrm{Lin}(U,V)$ of order $rk$, which conveys a substantial improvement on the degree of the potential $\cB(D)$ associated to $\cA(D)$ obtained by Raiț{\u{a}} for symbols over $\R^d$ (which has order $2rk$). It should be remarked, however, that in some cases our potential $\cB(D)$ may convey more  equations ($\dim U > \dim V$). In practice, one can still argue this is a sensitive gain given that $B(\xi)$ acts linearly on the $U$-variable, while the order $k'$ of $\cB(D)$ considerably increases the non-linearity of the $\xi$-variable of the symbol $B(\xi)$. Despite our improvement on the order of $\cB(D)$, our explicit construction may not attain the minimal possible order. Therefore, in Proposition~\ref{prop:construction} we give another (abstract) construction of an \emph{optimal potential}  operator $\cB(D)$, which extends the optimal construction  for elliptic operators by Van Schaftingen~\cite{schaft}. Finally, since most of our constructions are valid for symbols over $\K^d$, we also establish in Theorem~\ref{thm:C} the existence of an exact homology for symbols over $\R^d$ with constant complex-rank.

Further, we discuss the homological properties of differential complexes associated with constant-rank operators for spaces of functions defined in the full space $\R^d$. In particular, we prove that  a generalized Poincar\'e lemma holds for a class of  \emph{zero mean-value} Schwartz functions $v \in \mathcal S(\R^d;V)$. This and a simple duality argument, allows us to give a direct extension of the Poincar\'e lemma for constant-rank operators on several spaces of \emph{homogeneous distributions} (see Theorems~\ref{thm:poincare} and~\ref{cor:PL}). Thus, extending the the Poincar\'e Lemma's established in~\cite[Theorem 3.5]{G} and in~\cite[Proposition 3.16]{gr1}. As a byproduct of this result, we show (see Corollary~\ref{cor:sob}) that if $m \in \mathbb Z$, $p \in (1,\infty)$ and   $v$ is a (class) distribution  in the homogeneous Sobolev space $\dot W^{m,p}(\R^d;V)$ satisfying
\[
    \cA v = 0 \quad \text{in the sense of distributions on $\R^d$,}
\]
then there exists a constant $C = C(m,p,A)$ and a map $u \in \dot W^{m+k,p}(\R^d;U)$  such that 
\[
\cB u = v \quad \text{and} \quad \|u\|_{\dot W^{m+k',p}} \le C \|v\|_{\dot W^{m,k}}\,,
\]
where $k'$ is the order of $\cB(D)$.

Lastly, we make the observation (see Lemma~\ref{thm:last}) that the existence of potentials, when restricted to spaces of periodic maps $C^\infty(\mathbb T^d;V)$ in dimensions $d \ge 2$, is equivalent to a \emph{strictly weaker} property than~\eqref{eq:exact}. Exploiting that our symbolic construction works for arbitrary fields, we prove in Theorem~\ref{thm:periodic} that the \emph{integer constant-rank} property
\[
    \rank A(m) = r \qquad \text{for all $m \in \mathbb Z^d = \{0\}$,}
\]
is a \emph{sufficient} condition for the existence of a potential $\mathcal B(D)$ when restricted to function spaces of smooth periodic maps with zero mean-value.

\section{Homological properties of polynomial symbols}

The homological properties of differential operators that we study in this paper are defined purely through their (principal) symbol, which is a homogeneous map depending polynomially on $\xi \in \mathbb{R}^{d}$. For this reason, we first focus on such maps. Application of this to the theory of differential operators will be discussed in Section~\ref{sec:PDE}. In this first section, we allow a little more flexibility and, in particular, we will consider fields other than that of real numbers (cf. Remark \ref{rmk:other fields}).\\

\noindent\textbf{Notation.} Let $\K$ be a field of characteristic $\neq 2$, and let $V, W$ be finite-dimensional $\K$-vector spaces, of dimensions $\dim_{\K}(V) = N$ and $\dim_{\K}(W) = M$. A symbol 
\[
A(\xi) : V \to W, \qquad \xi \in \K^d,
\]
is a $\mathrm{Lin}(V,W)$-valued polynomial on $\xi$ so that each $A(\xi)$ is a linear map from $V$ to $W$.  Let us choose bases $e_{1}, \dots, e_{N}$ of $V$ and $f_{1}, \dots, f_{M}$ of $W$, respectively, so we can think of  $A(\xi)$ as the $M \times N$ matrix
\[
    A(\xi) = a^i_j(\xi), \qquad \xi \in \K^d,
\]
where the coordinates
\[
    a^i_j \in \K[\xi_1,\dots,\xi_d], \qquad i = 1,\dots,M, \; j = 1,\dots,N,
\]
are homogeneous polynomials of the same order. We will denote the columns of the matrix $A(\xi)$ by $a_{1}(\xi), \dots, a_{N}(\xi)$. Given that a symbol takes values in a space of matrices with $\K$-coefficients, the integer-valued quantity
\[
    \rank_\K A(\xi) = \dim_\K A(\xi)[V]
\]
is well-defined for all $\xi \in \K^d$. 
We say that a symbol $A(\xi)$ has constant $\rank$ if there exists a non-negative integer $r$ such that
\begin{equation}
        \rank_\K A(\xi)  = r \quad \text{for all $\xi \in \K^d - \{0\}$}.
\end{equation}
We shall often simply write $\rank_\K A = r$.

Exterior products and exterior powers will appear throughout the article. We recall that, if $\K$ is a field of characteristic $\neq 2$, and $W$ is an $M$-dimensional vector space then, for $r \leq M$ the $r$-fold exterior product
\[
\bigwedge\!^{r}\,W
\]
is the subspace of the $r$-fold tensor product $W^{\otimes r}$ that is spanned by elements of the form
\[
m_{1}\wedge \cdots \wedge m_{r} := \sum_{\sigma \in S_{r}}\operatorname{sign}(\sigma)m_{\sigma(1)}\otimes \cdots \otimes m_{\sigma(r)}
\]
where $\sigma \in S_r$ is a permutation of the set $\{1, \dots, r\}$ with sign $\operatorname{sign}(\sigma) \in \{\pm 1\}$.\footnote{If the characteristic of $\K$ is $2$, then there are two distinct possible ways to define the exterior powers. To not go into such matters, we avoid this.} It is known that $\bigwedge^r W$ is a vector space of dimension $\binom{M}{r}$, with a basis given by
\[
w_{i_{1}}\wedge \cdots \wedge w_{i_{r}}
\]
where $i_{1} < \cdots < i_{r}$ and where $w_1, \dots, w_{M}$ is itself a basis of $W$. It is a classical linear-algebraic fact that $m_1 \wedge \cdots \wedge m_r = 0$ if and only if the set $\{m_1, \dots, m_r\}$ is linearly dependent or, equivalently, if $\dim_{\K}\operatorname{span}\{m_1, \dots, m_r\} < r$. We will use this fact repeatedly throughout this work.

\subsection{Construction of an annihilator}

We maintain the notation of the section above. In particular, $A(\xi): V \to W$ is a homogeneous, degree $k$ symbol with constant rank $\rank_{\K} A = r$. We introduce the vector space
\begin{equation}\label{eqn:F}
X:= \left(\bigwedge\!^{r+1}\, W\right)^{\binom{N}{r}}
\end{equation}
which has dimension $\binom{M}{r+1}\binom{N}{r}$.\footnote{For consistency, we convene that $\binom{M}{M + 1} = 0$. This applies when $r =\dim W$, in which case $X = \{0\}$.} 
We are now in position to introduce our main explicit construction, which is a symbol $Q(\xi): W \to X$ depending on $\xi_1, \dots, \xi_d \in \K$ in a polynomial manner as follows:
\begin{equation}\label{eqn:Q}
 Q(\xi)(w) \coloneqq (a_{i_1}(\xi)\wedge \cdots \wedge a_{i_r}(\xi) \wedge w)_{1 \leq i_{1} < \cdots < i_{r} \leq N}\,.
\end{equation}

The following result shows that $Q(\xi)$ is an exact algebraic annihilator of the symbol $A(\xi)$:
\begin{lemma}\label{thm:Q} Let $A(\xi) : V \to W$ be a homogeneous symbol of degree $k$ on $\K^d$ with constant rank
\[
    \rank_\K A(\xi) = r \qquad \text{for all $\xi \in \K^d - \{0\}$.}
\]
Then, $Q(\xi) : W \to X$ is a homogeneous symbol on $\K^d$ satisfying the following properties:
\begin{enumerate} \setlength\itemsep{3pt}
    \item  If $r  < \dim W$, then the order of $Q(\xi)$ is $r k$.
    \item If $r = M$, then $X = \{0\}$ and $Q(\xi)$ is the zero operator.
    \item In either case,
        \[
        \im A(\xi) = \ker Q(\xi) \quad \text{for all non-zero} \; \xi \in  \K^d,
    \]
  
\end{enumerate}
\end{lemma}
\begin{proof}
Properties (1) and (2) are immediate from the definition,  
so we only need to show Property (3) for the non-trivial case when $r < \dim W$. We separate this into proving two set inclusions. Let $\xi \in \K^d$ be non-zero.

First, we prove that if $w \in \im A(\xi)$, then $Q(\xi)w = 0$. By linearity, it suffices to show this for $w = a_i(\xi)$, $i = 1, \dots, M$. Let $I = \{i_1, \cdots, i_r\}$ be a strictly ordered subset of $\{1,\dots,N\}$. Since $\rank A(\xi)  \leq r$, either there is a repeated element in  $\{a_{i_1}(\xi),\dots,a_{i_r}(\xi),a_i(\xi)\}$ or this set is  linearly dependent. We get
\[
    a_{i_{1}}(\xi) \wedge \cdots \wedge a_{i_{r}}(\xi) \wedge a_{i}(\xi) = 0\; \forall I \quad \Longrightarrow \quad Q(\xi)a_{i}(\xi) = 0.
\]
This proves that $\im A(\xi) \subset \ker Q(\xi)$ for all $\xi \in \K^d - \{0\}$.

Now take $w \in \ker Q(\xi)$. Since  $\rank A(\xi)  \geq r$, there exists a subset $I = \{i_1 < \cdots < i_{r}\} \subseteq [1,N]$ such that the set $\{a_{i_1}(\xi), \dots, a_{i_{r}}(\xi)\}$ is linearly independent. But by our assumption on $w$, $a_{i_{1}}(\xi) \wedge \cdots \wedge a_{i_{r}}(\xi) \wedge w = 0$, so $\{a_{i_1}(\xi), \dots, a_{i_{r}}(\xi), w\}$ is linearly dependent. This means that $w$ belongs to the span of $a_{i_{1}}(\xi), \dots, a_{i_{r}}(\xi)$, and thus to the image of $A(\xi)$. This finishes the proof.
\end{proof}

\subsection{The rank over the ring of polynomials} Following ~\cite{sturmfels3}, we may consider a symbol $A(\xi)$ as a map acting on vector-valued polynomials, that is, we consider its  lifting
\[
\xymatrix@1{{\mathbb K[\xi]^V}\ar[r]^-{A(\xi)} \; & \;
	{\mathbb K[\xi]^W}}\,.
\]
 If we do not assume that $\rank_\K A$ is constant, we can still perform the same construction but now taking the rank 
\[
\rank_{\K[\xi]}A(\xi)
\]
which is \emph{always} well-defined. Note that, for every specialization of the variables $\xi \in \K^d$ we have
\[
\rank_\K A(\xi) \leq \rank_{\K[\xi]}A(\xi)
\]
and we obtain equality outside of an algebraic set $\mathcal{V} \subseteq \K^n$, given by the vanishing of the $r \times r$-minors of the matrix $A(\xi)$. If $\xi \in \K^n - \mathcal{V}$, then it still holds that $\im A(\xi) = \ker Q(\xi)$. If, on the other hand, $\xi \in \mathcal{V}$, then $Q(\xi) = 0$ (cf.  \cite[Theorem 1.3]{harkonen}).

\subsection{Homological properties of symbols} The following result says that, if we specialize to the case $\K = \R,\C$, then the existence of an annihilator $Q(\xi)$ characterizes all  homogeneous symbols of constant rank $A(\xi)$. See Remark \ref{rmk:semicontinuity of rank} below for a discussion on dependence of the result on the choice of fields $\R$, $\C$.

\begin{theorem}\label{thm:algebra} Let $\mathbb K = \R,\C$, and let $A(\xi): V \to W$ be a homogeneous symbol on $\K^d$. The following are equivalent:  

\begin{enumerate}\setlength\itemsep{3pt}
    \item $\rank_\K A = r$. 
    \item there exists a symbol complex 
\[
\xymatrix@1{{U}\ar[r]^-{B(\xi)} \; & \;
	{V}\ar[r]^-{A(\xi)} \; & \;  {W}\ar[r]^{Q(\xi)} \; & \;  {X}},
\]
where both $B(\xi)$ and $Q(\xi)$ are homogeneous symbols  and 
\[
    \im B(\xi) = \ker A(\xi) \quad \text{and} \quad \im A(\xi) = \ker Q(\xi)
\]
for all $\xi \in \mathbb K^d - \{0\}$. 
\end{enumerate}

Moreover, if (1) is satisfied then in (2) we can always take $B(\xi)$ and $Q(\xi)$ homogeneous of order $rk$.
\end{theorem}

\begin{remark}\label{rmk:semicontinuity of rank}
 Even though for the applications we consider here it is enough to consider $\K = \R$ or $\C$, the reader should be aware (cf. Lemma~\ref{thm:Q}) that the proof of (1) $\Rightarrow$ (2) in Theorem~\ref{thm:algebra} is unchanged if, instead, we take $\K$ to be \emph{any} field of characteristic different from $2$. Note that polynomial differential operators make sense over any field: we have the algebra of polynomial differential operators
$$
D(\K^n) := \K\langle x_1, \dots, x_n, \partial_1, \dots, \partial_n\rangle / ([x_{i}, x_{j}] = 0, [\partial_{i}, \partial_{j}] = 0, [\partial_{i}, x_{j}] = \delta_{ij}).
$$
The symbol map is defined using the so-called \emph{Bernstein filtration} on this algebra, and $D(\K^n)$ contains the subalgebra of constant-coefficient differential operators, which is generated by $\partial_1, \dots, \partial_n$ and is known to be a polynomial algebra in $n$ variables. See, e.g., \cite{borel, HTT}. \end{remark}

\begin{proof}
That (2) implies (1) follows directly from the rank-nullity theorem and the lower semicontinuity of the rank as follows: Firstly, (2) implies that $\rank_\K A(\xi)$ is an integer-valued continuous function of $\xi \in \K^{d} - \{0\}$. When $d > 1$ or $\K = \C$, $\K^{d} - \{0\}$ is connected, so $\rank_{\K} A(\xi)$ is constant. When $d = 1$ and $\K = \R$, the set $\R -\{0\}$ is not connected. However, we can still conclude that $\rank_\R A(\xi)$ is constant on $\R_{>0}$ and on $\R_{<0}$. Now note that, by homogeneity, $A(-\xi) = (-1)^{k}A(\xi)$. Thus $\rank_{\R}A(\xi)$ is a fortiori constant on $\R-\{0\}$.

To see that (1) implies (2), we shall appeal to the construction of the previous theorem so that $Q(\xi)$ is precisely the operator constructed there, i.e., $X$ as in eq. \eqref{eqn:F} and $Q$ as in eq. \eqref{eqn:Q}. To construct $B(\xi)$ we argue by duality, so we need to set-up some notation. For the rest of this proof, if $A: V \to W$ is an operator, then $A^{*}: W^{*} \to V^{*}$ is its dual map, where as usual $V^{*}$ is the vector space of linear functionals on $V$. Moreover, if $Y \subseteq V$ is a subspace, we denote by $Y^{\perp} := \{\varphi \in V^{*} \mid \varphi(y) = 0 \; \text{for every} \; y \in Y\}$. It is easy to see that $\im A(\xi)^* = \ker A(\xi)^\perp$, and it follows that the dual $A(\xi)^*$ is also a symbol of order $k$ and constant rank $r$. Thus, we may apply the previous theorem to $A(\xi)^*$ to find a symbol $B(\xi)^*$ of order $rk$ and rank $\dim(V) - r$ satisfying $\im A(\xi)^* = \ker B(\xi)^*$ for all $\xi \in \R^d - \{0\}$. Dualizing this identity once more and writing $B(\xi) \coloneqq B(\xi)^{**}$, we deduce that
$B(\xi)$ is of order $rk$, of constant rank $\dim(V) - r$ and satisfies $\im B(\xi) = \ker A(\xi)$. 
\end{proof}

\begin{remark}
Note that (2) $\Rightarrow$ (1) in Theorem~\ref{thm:algebra} \emph{crucially} uses that $\K = \R$ or $\C$, for otherwise we cannot use a continuity argument to conclude that the rank of $A(\xi)$ is constant (compare this with Remark~\ref{rmk:other fields} below).
\end{remark}

\subsection{Real symbols with constant rank over $\C$}Since it has been the object of both classical and recent developments  in computational commutative algebra \cite{sturmfels,harkonen} and PDE theory~\cite{ADN1,ADN2,anna,slicing,diening,gmeineder,smith1,smith2}, we also discuss the homological properties of symbols over $\R$ with constant-rank over $\C$. Recall that if $V$ is a real vector space, its complexification   is defined to be
\[
V_{\C} \coloneqq \C \otimes_{\R} V = V \oplus \mathrm i V
\]
where the last decomposition is only as real vector spaces. If $f: V \to W$ is a linear map of real vector spaces, its complexification is
\[
f_{\C} := \operatorname{id}_{\C} \otimes_{\R} f: V_{\C} \to W_{\C}.
\]
In layperson's terms, $f_{\C}(v_1 + \mathrm i v_2) = f(v_1) + \mathrm i f(v_2)$. Note that $f_{\C}$ is clearly a linear map of complex vector spaces.

\begin{definition}
Given a symbol $A(\xi) : V \to W$ on $\R^d$, we define its complexification to be the symbol $A(\xi)_\C : V_\C \to W_\C$, 
where $A(\xi)_\C$ is considered as  a polynomial of complex variables. 
\end{definition}
We have the following:

\begin{theorem}\label{thm:C}
Let $A(\xi) : V \to W$ be a homogeneous symbol on $\R^d$. The following are equivalent:
\begin{enumerate}
    \item $\rank_\C A(\xi)_\C$ is constant on $\C^d - \{0\}$. 
\item There exists   a symbol complex (on $\R^d$)
\[
\xymatrix@1{{U}\ar[r]^-{B(\xi)} \; & \;
	{V}\ar[r]^-{A(\xi)} \; & \;  {W}\ar[r]^{Q(\xi)} \; & \;  {X}}
\]
where both $B(\xi)$ and $Q(\xi)$ are homogeneous (both have real coordinate coefficients) and satisfy the exactness properties:
\[
    \im B(\xi)_\C = \ker  A(\xi)_\C \quad \text{and} \quad \im A(\xi)_\C = \ker Q(\xi)_\C
\]
for all $\xi \in \mathbb C^d - \{0\}$.
\end{enumerate}
\end{theorem}
\begin{proof}
It is easy to see that the complexification of symbols  commutes with the construction of the operators $Q(\xi)$ and $B(\xi)$, from where the result follows.
\end{proof}

\begin{remark}\label{rmk:other fields}
The constant-rank property is not an invariant across distinct fields. Take, for instance the Cauchy--Riemann equations 
\[
    A(D)u = (\partial_1 u_1 - \partial_2 u_2,\partial_1 u_2 + \partial_2 u_1), \qquad u : \R^2 \to \R^2. 
\]
Its associated principal symbol is the conformal matrix field $A(\xi) = (\xi,\xi^\perp)$. 
Evidently, $A(\xi)$ is invertible for all $\xi \in \R^2$ since its determinant is precisely $|\xi|^2$. However, as a complex map it is not always invertible. Indeed, its determinant is $\xi_1^2 + \xi_2^2$, which is a polynomial with non-trivial zeroes in $\C^2$. 
\end{remark}

\subsection{Regularity properties of the Moore-Penrose symbol} Throughout this section and unless otherwise explicitly stated, we assume $\K = \R$. The Moore--Penrose inverse of $M \in \mathrm{Lin}(V,W)$ is the unique linear map $M^\dagger : W \to V$ defined by the fundamental property:
\begin{equation}\label{eq:moorepenrose}
    M^\dagger M = \mathrm{proj}_{\ker M^\perp}.
    \end{equation}
Here, the orthogonal space to the kernel  $(\ker M)^{\perp}$ is taken with respect to the usual inner product on $V \cong \R^{N}$. Given a symbol $A(\xi) : V \to W$ on $\R^d$, we may define the Moore-Penrose inverse of $A(\xi)$ as the unique map $A(\xi)^\dagger : W \to V$ satisfying
\[
    A(\xi)^\dagger \circ A(\xi) = \operatorname{proj}_{(\ker A(\xi))^\perp} \qquad \text{for every $\xi \in \R^d - \{0\}$.}
\]
Using the same ideas that motivated the construction of $Q(\xi)$ in Theorem \ref{thm:Q}, we obtain an immediate proof of the following fact, see e.g., \cite{raita2018potentials}. (See Remark \ref{rmk:MP C} below about the properties of the Moore--Penrose pseudoinverse map for symbols on $\C^d$.)

\begin{proposition}\label{prop: proj analytic}
Let $A(\xi) : V \to W$ be a homogeneous symbol on $\R^d$ satisfying the constant-rank property
\[
    \rank A(\xi) = r \qquad \text{for all $\xi \in \R^d - \{0\}$.}
\]
Then, the projection
\[
    \xi \mapsto \pi(\xi) \coloneqq \mathrm{proj}_{\ker A(\xi)^\perp}
\]
is rational and homogeneous of degree zero on $\mathbb R^d - \{0\}$. In particular, the Moore--Penrose pseudoinverse map
\[
    \xi \mapsto  A(\xi)^\dagger
\]
is rational and homogeneous of degree $-k$ on $\mathbb R^d - \{0\}$.
\end{proposition}
\begin{proof}
Upon taking adjoints, the first statement is equivalent to checking that the map $\varphi: \R^d - \{0\} \to \operatorname{Gr}(r, W) : \xi \mapsto \im A(\xi)$ is rational. Here, $\operatorname{Gr}(r, W)$ is the Grassmannian of $r$-dimensional subspaces in $W$, which is identified with an algebraic subvariety of the projective space $\mathbb{P}(\bigwedge\!^{r}W)$ by means of the Pl\"ucker embedding \cite{manivel, plucker}. The case where $A(\xi)$ is injective is clear, for the map $\varphi$ can be decomposed as $\xi \mapsto a_{1}(\xi) \wedge \cdots \wedge a_{r}(\xi)$ followed by the projection $(\bigwedge\!^{k}W - \{0\}) \to \mathbb{P}(\bigwedge\!^{k}W)$, both of which are rational. For the general case, we work locally. Let $\xi \in \R^d - \{0\}$. We know that there exist $i_{1}, \dots, i_{k}$ such that $a_{i_{1}}(\xi), \dots, a_{i_{r}}(\xi)$ are a basis for $\im A(\xi)$. But, since this is equivalent to the nonvanishing of a minor of $A(\xi)$, the same is true for every $\xi'$ in a neighborhood of $\xi$. We can then run the same argument as in the injective case, with $1, \dots, k$ replaced by $i_{1}, \dots, i_{k}$, to see that the map $\varphi$ is rational in a neighborhood of $\xi$ and is therefore rational everywhere on $\R^d - \{0\}$. Finally, since $A(\xi)$ is homogeneous we have that $\ker A(\xi) = \ker A(\lambda \xi)$ for every $\lambda \in \R$ and every $\xi \in \R^{d}$, from where homogeneity of degree $0$ follows immediately for $\pi(\xi)$. 

Let us now prove the statement on the Moore--Penrose inverse map. We have $A(\xi)^\dagger A(\xi) = \pi(\xi)$. We record that both $A(\xi)$ and $\pi(\xi)$ are matrices whose entries belong to the field of rational functions $\mathbb R(\xi_1, \dots, \xi_d)$. Expanding the product \eqref{eq:moorepenrose} keeping the entries of $M^{\dagger}$ unknown, we see that the entries of $M^{\dagger}$ solve a linear system of equations over the field $\mathbb R(\xi_1, \dots, \xi_d)$. Thus, its entries also belong to the field $\mathbb R(\xi_1, \dots, \xi_d)$. The claim about the degree of $M^{\dagger}$ is clear.
\end{proof}

\begin{remark}\label{rmk:MP C}
If, instead of taking $\K = \R$ we take $\K = \C$, then the map $\varphi: \C^{d} - \{0\} \to \mathrm{Gr}(r, W)$ is still rational, with the same proof as in that of Proposition \ref{prop: proj analytic}. However, the usual Hermitian form on $\C^N$ involves taking complex conjugates on one of its entries, so we obtain that $\xi \mapsto \pi(\xi)$ and $\xi \mapsto A(\xi)^{\dagger}$ are are still rational when considered as functions of $\xi$ and its conjugate $\overline{\xi}$. 
\end{remark}

\section{Homological properties of differential operators}\label{sec:PDE}

\subsection{Background theory}In order to lift the homological properties of the symbol complex (algebraic framework) to its associated differential complex (functional setting), we need to introduce a suitable space of functions. Let us recall that the space of Schwartz maps 
\[
\mathcal S(\R^d) \coloneqq \left\{f \in C^\infty(\R^d) : \sup_{\alpha,\beta} \|x^\alpha \partial^\beta f(x)\|_\infty < \infty \right\}
\]
is the space of smooth maps on $\R^d$ whose derivatives of all orders decay faster than any polynomial rate at infinity. We consider its subspace
\[
     \dot{\mathcal S}(\R^d) \coloneqq \left\{f \in \mathcal S(\R^d) : (\partial^\alpha\widehat f)(0) = 0 \; \text{for every multi-index $\alpha$} \right\},
\]
where
\[
    \widehat f(\xi) = \mathcal Ff( \xi)  \coloneqq  \int_{\R^d} \mathrm{e}^{2\pi \mathrm{i} (\xi \cdot x)} f(x) \, dx,
\]
is the Fourier transform of $f$. The space $\dot{\mathcal S}(\R^d)$ inherits the same topology of $\mathcal S(\R^d)$ and with this topology it is a closed subspace. In particular, we have (see p.10 in~\cite{Triebel})
\[
    \dot{\mathcal S}(\R^d) = \set{\varphi \in \mathcal S(\R^d)}{\|\varphi\|_{k}^* < \infty, k \in \mathbb N_0},
\]
where (see~\cite[p.10]{Triebel})
\[
    \|\varphi\|_{k}^* = \sup_{\substack{\xi \in \R^d,\\ 0\le |\alpha| \le k}}  \left(|\xi|^k + |\xi|^{-k} \right) |D^\alpha \widehat \varphi (\xi)|, \qquad k \in \mathbb N_0.
\]

Note that if $f \in \dot{\mathcal S}(\R^d)$, then 
\begin{equation}\label{eq:ave}
    \int_{\R^d} p(x)  f(x) \, dx =  0,
\end{equation}
for all polynomials $p \in \R[x_1,\dots,x_d]$. As usual we write  $\mathcal S'(\R^d)$ to denote space of tempered distributions, which is the topological dual of $\mathcal S(\R^d)$. The Fourier transform is therefore extended to $\mathcal S'(\R^d)$ by duality.  
The space $\dot{\mathcal S}'(\R^d)$ of \emph{homogeneous tempered distributions} is defined as the continuous dual space of $\dot{\mathcal S}(\R^d)$. Notice that~\eqref{eq:ave} and the Hahn--Banach theorem allow one to identify $\dot{\mathcal S}'(\R^d)$ with the quotient space $\mathcal S'(\R^d)/\R[x_1,\dots,x_d]$ of tempered distributions modulo polynomials. In particular, $L^p$-spaces are subspaces of  homogeneous distributions, that is,
\[
    L^p(\R^d) \cap \dot{\mathcal S}'(\R^d) = L^p(\R^d).
\]

 It is well-known that $ \mathcal F$ defines a linear isomorphism from $\mathcal S(\R^d;\C)$ into itself, which by duality also extends to an isomorphism from $\mathcal S'(\R^d;\C)$ into itself. We write $(\frarg)^\vee$ to denote the inverse of $\mathcal F$. 
Appealing to the Taylor expansion of $\widehat f$ at $0$, it is immediate to verify that for every $\sigma \in \mathbb R$, the $\sigma$-Riesz potential convolution operator
\[
    I_\sigma f \coloneqq (|\xi|^\sigma \widehat f)^\vee, 
\]
defines an isomorphism from $\dot{\mathcal S}(\R^d;\C)$ into itself. Indeed, by the Leibniz rule it follows that if $|\sigma| \le m \in \mathbb N$, then
\begin{align*}
    \|I_\sigma f\|_k^* & \le C_{d,k,\sigma} \|f\|_{2k + m}^*, \qquad k \in \mathbb N_0
\end{align*}
Once again, by duality, the $\sigma$-Riesz potential $I_\sigma$ extends to an isomorphism from $\dot{\mathcal S}'(\R^d;\C)$ into itself. These considerations extend in a natural way to $\mathcal S(\R^d;V)$ and $\mathcal S'(\R^d;V)$, the respective spaces of $V$-valued Schwartz and tempered distribution spaces. We remind the reader that if $V$ is a $\C$-space and $f \in \mathcal S(\R^d;V)$, then $\widehat f \in \mathcal S(\R^d;V)$. If, on the other hand, $V$ is only an $\R$-vector space, then $\mathcal S(\R^d;V) \embed \mathcal S(\R^d; V_\C)$, so in this case we naturally have the Fourier transform $\widehat f \in \mathcal \mathcal S(\R^d; V_\C)$.


\subsection{Homology for homogeneous spaces} 

The following result is a full-space analog  of~\cite[Lemma~2]{raita2018potentials}, where a similar result has been established for functions defined over the flat torus $\R^d / \mathbb Z^d$.

\begin{theorem}\label{thm:poincare} Let $\cA(D)$ be a constant coefficient $k$\textsuperscript{th} order operator on $\R^d$ from $V$ to $W$. The following are equivalent:
\begin{enumerate}[(1)]
    \item $\rank A(\xi)$ is constant on $\R^d - \{0\}$
\item there exists a complex of differential operators 
\[
\xymatrix@1{{\mathcal D'(\R^d;U)}\ar[r]^-{\cB(D)} \; & \;
	{ \mathcal D'(\R^d;V)}\ar[r]^-{\cA(D)} \; & \;  {\mathcal D'(\R^d;W)}\ar[r]^{\cQ(D)} \; & \;  {\mathcal D'(\R^d;X)}},
\]
which restricts to an exact differential complex 
\begin{equation*}
\xymatrix@1{{\dot{\mathcal{S}}(\R^d;U)}\ar[r]^-{\cB(D)} \; & \;
	{\dot{\mathcal{S}}(\R^d;V)}\ar[r]^-{\cA(D)} \; & \;  {\dot{\mathcal{S}}(\R^d;W)}\ar[r]^{\cQ(D)} \; & \;  {\dot{\mathcal S}(\R^d;X)}},
\end{equation*}
that is, 
\[
     \im \cB(D)|_{\dot{\mathcal S}} = \ker \cA(D)|_{\dot{\mathcal S}} \quad \text{and} \quad \im \cA(D)|_{\dot{\mathcal S}} = \ker \cQ(D)|_{\dot{\mathcal S}}.
\]
In particular, it also restricts to an exact differential complex
\[
\xymatrix@1{{\dot{\mathcal{S}}'(\R^d;U)}\ar[r]^-{\cB(D)} \; & \;
	{\dot{\mathcal{S}'}(\R^d;V)}\ar[r]^-{\cA(D)} \; & \;  {\dot{\mathcal{S}'}(\R^d;W)}\ar[r]^{\cQ(D)} \; & \;  {\dot{\mathcal S'}(\R^d;X)}},
\]
 on spaces of homogeneous tempered distributions. 
\end{enumerate}
\end{theorem}

\begin{proof} That $(2)$ implies $(1)$ follows from the last assertion in (2), a standard localization argument, an application of the Fourier transform and  Theorem~\ref{thm:algebra}. We now prove that $(1)$ implies $(2)$.  
Let $B( \xi),Q( \xi)$ and $U,X$ be the elements of the symbol complex given in Theorem~\ref{thm:algebra}. As before, we write $\cB(D),\cQ(D)$ to denote their associated operators, which are well defined on spaces of distributions. A standard localization and mollification argument, together with an application of the Fourier transform and (2) in Theorem~\ref{thm:algebra}, gives $ \im  \cB(D)  \subset  \ker \cA(D)$ and $\im \cA(D)  \subset  \ker \cQ(D)$. This proves that the sequence composed by $\cB(D),\cA(D),\cQ(D)$ defines a differential complex (for all sub-spaces of distributions that are invariant under differentiation).

In light of a duality argument, to prove the second statement  it suffices to prove the statement for  the differential complex over $\dot{\mathcal S}$-spaces of functions. We need to show that $\ker\cA(D) \subseteq \im\cB(D)$, for the inclusion $\ker \cQ(D) \subseteq \im\cA(D)$ is obtained analogously. The proof follows closely the some of the concepts already contained in~\cite{G,FM99,Mur81,raita2018potentials}. 
In the following, we will use the simplified notation $A(\xi)a$ to denote $A(\xi)_\C[a]$ when $V$ is an $\R$-space and $a \in V_\C$. 
Let us fix $v \in \ker \cA(D)$. 
Applying the Fourier transform to $v$ we find that
\[
    0 = \mathcal F (\cA v)(\xi) = (2\pi\mathrm{i})^k A(\xi)  \widehat v(\xi),
\]
which by construction implies that 
\begin{equation}\label{eq:im}
\widehat v(\xi) \in \im B(\xi)_\C \qquad \text{for all $\xi \in \R^d - \{0\}$}.
\end{equation}
Consider the tempered distribution $u \in \mathcal S'(\R^d;U_\C)$ defined by the Fourier transform 
\begin{equation}\label{eq:link}
    \widehat u(\xi) \coloneqq (2 \pi \mathrm i)^{-rk} B(\xi)^\dagger \widehat v(\xi) = (2 \pi \mathrm i)^{-rk} |\xi|^{-k} (M \widehat v)(\xi),
\end{equation}
where $M$ is the zero-homogeneous profile of $B^\dagger$, which depends smoothly on $\xi$ in the punctured space $\R^d - \{0\}$. Since $v \in \dot{\mathcal S}(\R^d;V)$,  it follows that $\tilde u = ((2 \pi \mathrm i)^{-rk} M\widehat v)^\vee$ belongs to $\dot{\mathcal{S}}(\R^d;U_\C)$. Notice that if $V$ is an $\R$-vector space, then  $\widehat u$ is a Hermitian function.  Given that $M$ is zero-homogenenous $(V \otimes W^*)$-valued map, in this case we also have that $M\widehat v$ is  Hermitian. From this analysis, we infer that $\tilde u   \in \dot{\mathcal{S}}(\R^d;U)$, regardless of $V$ being an $\R$-space or $\C$-space. By the discussion above on the properties of the Riesz potential, we conclude that  $u = I_{-k'} \tilde u \in \dot{\mathcal S}(\R^d;U)$.

We are left only to verify that indeed $\cB u = v$, or equivalently, that 
\[
    (2 \pi \mathrm i)^{rk} B(\xi) \widehat u(\xi)  
    = \widehat v(\xi), \qquad \text{for all $\xi \in \R^d - \{0\}$.} 
\]
This follows easily from \eqref{eq:im} and the fact that 
\[
  B(\xi) \circ B(\xi)^\dagger = [B(\xi)^\dagger \circ B(\xi)]^t = \operatorname{proj}_{\im B(\xi)}
\]
for all nonzero $\xi \in \R^d$. 
This proves that $\ker\cA(D) \subseteq \im\cB(D)$ as desired. 
\end{proof}

The fact that the homology of the differential complex is trivial conveys the validity of a \emph{generalized} Poincar\'e Lemma for homogeneous Besov and Triebel--Lizorkin spaces. For a precise definition and properties of these spaces, we refer the reader to Triebel's book~\cite[Ch. 2]{Triebel}. More precisely, we obtain the following full-space generalized Poincar\'e lemma:

\begin{theorem}\label{cor:PL} Let $s \in \R$ and let $p,q \in (0,\infty)$. Let $\cA(D)$ be a constant coefficient $k$\textsuperscript{th} order operator satisfying the constant-rank property 
\[
    \rank A(\xi) = r \quad \text{for all non-zero $\xi \in \R^d$}.
\]
Let $X = \{B,F\}$ and let $v \in  \dot X^s_{p,q}(\R^d;V)$ be such that 
\[
 \mathcal A v  = 0 \qquad \text{in the sense of distributions.}
\]
Then, there exists a bounded linear map $T : \dot X^s_{p,q}(\R^d;V) \to \dot X^{s + k'}_{p,q}(\R^d;U)$ satisfying
\[
\cB (Tv) = v \quad \text{as distributions on $\R^d$,}
\]
where $k'$ is the order of $\cB(D)$.

Moreover,  the norm $\|T\|$ of $T$ depends solely on $d,s,p,q$ and $A(\xi) : V \to W$. 
\end{theorem}
\begin{proof} In light of~\eqref{eq:link} and standard duality arguments, it suffices to verify that
\[
    \|I_{-k'} (Mv)\|_{\dot X^{s + k'}_{p,q}} \lesssim_{d,s,p,q,A}  \|Mv\|_{\dot X^s_{p,q}}\lesssim_{p,A} \|v\|_{\dot X^s_{p,q}}
\]
for all $v \in \dot{\mathcal S}(\R^d;V)$. The first inequality follows directly from~\cite[Proposition 2.8]{Triebel} (here, we are using that $k'$ depends intrinsically on $A(\xi) : V \to W$). The  second one is a direct consequence of Mihlin's theorem~\cite[Theorem 5.2.2]{Triebel2} for such spaces, using that $M$ is a zero-order $L^p$-multiplier depending solely on $A(\xi) : V \to W$. 

\end{proof}

Lastly, we record a generalized Poincar\'e lemma for homogeneous Sobolev spaces, which extends the results contained in~\cite[Theorem 3.5]{G} and~\cite[Proposition 3.16]{gr1}. Let $m$ be a non-negative integer and let $p \in (1,\infty)$. The homogeneous Sobolev space $\dot W^{m,p}(\R^d)$ is the collection of  all $f \in \dot{\mathcal S}'(\R^d)$ such that
\[
    \|f\|_{\dot W^{m,p}} = \sum_{|\alpha| = m} \|\partial^\alpha f\|_{L^p} < \infty. 
\] 
Notice that $\dot W^{0,p}(\R^d) = L^p(\R^d)$. For negative $m$ we set $\dot W^{-m,p}(\R^d) = (\dot W^{m,p}(\R^d))'$.

\begin{corollary}\label{cor:sob}
Let $m \in \mathbb Z$. Let $v \in \dot W^{m,p}(\mathbb R^d;V)$ and further assume that
\[
    \cA v = 0 \quad \text{in the sense of distributions on $\mathbb R^d$.}
\]  
Then, there exists $u \in \dot W^{m + k',p}(\R^d;V)$ such that
\[
    \cB u = v \quad \text{as measurable maps on $\R^d$} 
\]
and satisfying the Sobolev estimate
\[
    \|u\|_{\dot W^{k' + m,p}} \le C \|v\|_{\dot W^{m,p}},
\]
where $k'$ is the order of $\mathcal B$ and $C$ depends only on $d,m,p$ and $A(\xi) : V \to W$.  

Moreover, the assignment 
\[
\dot W^{m,p}(\R^d) \to \dot W^{k' + m,p}(\R^d) : v \mapsto u
\]
is linear.
\end{corollary}
\begin{proof} First, we address the case when $m \ge 0$.
By the the previous corollary we know that if $v \in \dot F^m_{p,2}(\R^d)$, then, with $v = I_{-k'}(Mu)$, it holds
\[
    \|u\|_{\dot F^s_{m+k',2}} \lesssim_{d,s,m,A} \|v\|_{\dot F^s_{m,2}}.
\]
The assertion then follows directly from the facts that (see Theorem 5.2.3/1(ii) in~\cite[]{Triebel2} for the nontrivial case $m > 0$)
\[
L^p  \subset \dot F^0_{m,2}, \qquad  \dot F^m_{m,2} = \dot W^{m,p} \text{ for $m > 0$},
\]
and that $\|v\|_{\dot F^0_{m,2}} \sim \|\frarg\|_{L^p}$ on $L^p$ and $\|v\|_{\dot F^s_{m,2}} \sim \|\frarg\|_{\dot W^{m,p}}$ for $m > 0$. The case for $m < 0$ is similar, using that  the topological dual of $\dot F^s_{p,q}$ is isomorphic to $\dot F^{-s}_{p,q}$ (this follows directly from the usual $L^p$-duality and the way the norm is defined on these spaces, see, e.g.,~\cite[Ch. 2]{Triebel}). 
\end{proof}

\subsection{The homology for Schwartz functions}\label{sec:S} Let us give an example that shows that, in general, the homology of an differential complex (associated with an exact symbol complex)
\[
\xymatrix@1{{\mathcal{S}(\R^d;U)}\ar[r]^-{\cB(D)} \; & \;
	{\mathcal{S}(\R^d;V)}\ar[r]^-{\cA(D)} \; & \;  {\mathcal{S}(\R^d;W)}\ar[r]^{\cQ(D)} \; & \;  {\mathcal S(\R^d;X)}},
\]
may be non-trivial over spaces of rapidly decaying maps.\footnote{The observation that the homology is not trivial for spaces of Schwartz functions is due to Andr\'e Guerra, who pointed it out to the first author during a visit at Oxford University.}  Note that this corrects a minor oversight in the last assertion of \cite[Lemma 2]{raita2018potentials}.

\begin{example}Let $\cA(\partial)$ be the derivative operator acting on functions of one variable
\[
    \cA(\partial)(u) = \frac{d u}{dt}, \qquad u : \R \to \R.
\]
This defines an operator on $\R$, from $\R$ to $\R$, of rank $1$. In particular the symbol $A(t)$ is onto for all non-zero $t \in \R$ and therefore any homogeneous annihilator $Q(t)$ of $A(t)$ must be the zero polynomial. Note, however, that if we consider $\cA(\partial),\cQ(\partial)$ as operators $\cA(\partial),\cQ(\partial): \mathcal{S}(\R) \to \mathcal{S}(\R)$ then $\ker \cQ(\partial)/\im \cA(\partial) \neq \{0\}$. Indeed, the fundamental theorem of calculus implies that 
\[
    \int_{-\infty}^\infty \frac{d u}{dt}\, dt = \lim_{t \to \infty} u(t) - \lim_{t \to -\infty} u(t) = 0 \quad \text{for all $u \in \mathcal S(\R)$.}
\]
However, $\ker \mathcal Q(\partial) \equiv \mathcal S(\R)$, which contains functions with non-zero average so that $\ker\cQ(\partial)/\im\cA(\partial)$ is, in fact, infinite-dimensional. Thus, the conclusion of Theorem \ref{thm:poincare} is not valid when $\dot{\mathcal{S}}$ is replaced by $\mathcal{S}$.
\end{example}

\subsection{The homology for periodic maps} Instead of working on full space one may consider maps over the $d$-dimensional flat torus $\mathbb T^d = \R^d / \mathbb Z^d$. In practice, a map $f \in C^\infty(\mathbb T^d)$ can be identified with a $\mathbb Z^d$-periodic map in $C^\infty(\R^d)$. Moreover, such maps can be decomposed in Fourier series as
\[
    f(x) = \sum_{m \in \mathbb Z^d} (\mathfrak F f)(m) \, \operatorname{e}^{2\pi \textrm{i} m\cdot x}, \qquad x \in \mathbb T^d,
\]
where 
\[
    \mathfrak F f(m) = \int_{\mathbb T^d}f(y) \operatorname{e}^{-2\pi \textrm{i} m\cdot y}  \, dy
\]
denotes the Fourier coefficient at $m \in \mathbb Z^d$. Similarly to the properties of the Fourier transform, a map $f \in C(\mathbb T^d)$ is smooth if and only if its coefficients $|\mathfrak F f(m)|$ decay faster than any polynomial as $|m| \to \infty$ (see for instance Corollary 3.2.10 and Proposition 3.2.12 in~\cite{Grafakos}). Given the identification with periodic maps, it therefore makes sense to consider the action of $\mathcal A(D)$ on a map $v \in \mathcal D(\mathbb T^d;V)$. Notice that, in this case
\[
    \mathfrak F(\mathcal A v)(m) = (2\pi \mathrm i)^{k_A}  A(m) \,\mathfrak F v(m), \qquad m \in \mathbb Z^d.
\]
The corresponding space of homogeneous periodic maps is the space
\[
 {\mathcal D}_\texttt{\#}(\mathbb T^d) = \set{f \in \mathcal D(\mathbb T^d)}{\mathfrak Ff(0) = 0}
\]
of  periodic functions with zero mean on the unit cube.

The following result (and the example below) shows that there exist operators that possess an exact potential when acting on spaces of periodic maps, which however possess no exact potential when acting on functions of $\R^d$:
\begin{lemma}\label{thm:last}
Let $\mathcal A(D)$ be a constant coefficient homogeneous operator on $\R^d$ from $V$ to $W$.  The following are equivalent:
\begin{enumerate}
    \item There exists $\mathcal B(D)$ from $U$ to $V$ such that 
    \[
        \im B(m) = \ker A(m) \qquad \text{for all $m \in \mathbb Z^d - \{0\}$}.
    \]
    \item There exists $\mathcal B(D)$ from $U$ to $V$ such that the sequence 
    \[
\xymatrix@1{{\mathcal D_\texttt{\#}(\mathbb T^d;U)}\ar[r]^-{\cB(D)} \; & \;
	{ \mathcal D_\texttt{\#}(\mathbb T^d;V)}\ar[r]^-{\cA(D)} \; & \;  {\mathcal D_\texttt{\#}(\mathbb T^d;W)}},
\]
defines a  complex of differential operators satisfying
\[
    \im \mathcal B(D)  = \ker \mathcal A(D).
\]
\end{enumerate}
\end{lemma}
\begin{proof}
Let us prove that (1) implies (2). That $\mathcal B[\mathcal D_\texttt{\#}(\mathbb T^d;U)] \subset \ker \mathcal A \cap \mathcal D_\texttt{\#}(\mathbb T^d;V)$ follows from applying the Fourier coefficient decomposition and the set inclusion $\im B(m) \subset \ker A(m)$ for all nonzero $m \in \mathbb Z^d$. That the homology is, in fact, trivial follows from the other inclusion as follows: if $v \in \mathcal D_\texttt{\#}(\mathbb T^d;V)$ satisfies $\mathcal Av = 0$, then we may define a periodic $U$-valued map by setting
\[
    u(x) = \sum_{\mathbb Z^d - \{0\}} a_m \operatorname{e}^{2\pi \mathrm i m \cdot x},
\]
where
\[
    a_m = (2\pi \mathrm i)^{-k_B} B(m)^\dagger \mathfrak F v(m) \quad \text{for all $m \in \mathbb Z^d - \{0\}$.}
\]
First, we need to see that $u$ is well-defined. Recall that $v$ is smooth, so that its Fourier coefficients decay faster than any polynomial (see  Corollary 3.2.10 in~\cite{Grafakos}). Since $B(\frarg)^\dagger$ is (at worst) negatively homogeneous, it follows that the coefficients $|a_m|$ also decay faster than any polynomial. The  trigonometric sum defining $u$ is therefore well-defined, as it  is uniformly convergent. In particular, $u$ is a mean-value zero continuous map with $\mathfrak F u(m) = a_m$.  Proposition 3.2.12 in~\cite{Grafakos}  further implies that $u$ is smooth. 
Moreover, by assumption $\mathfrak F v(m) \in \ker A(m)$ for all $m \in \mathbb Z^d$. Therefore, by the identity of the Moore--Penrose inverse we get 
\[
    (2 \pi \mathrm i)^{k_B} B(m) \mathfrak F u(m) = B(m) B(m)^\dagger \mathfrak F v(m) = v(m) \quad \text{for all $m \in \mathbb Z^d - \{0\}$.}
\]
This proves that $\mathcal B u = v$ as desired. 

To see that (2) implies (1) we argue as follows. First, we observe that if $P \in \ker A(m)$ for some $m \in \mathbb Z^d - \{0\}$, then any constant-polar smooth map of the form
\[
   P  \varphi(x \cdot m), \qquad \varphi \in \mathcal D_\texttt{\#}(\mathbb T)
\]
is annihilated by $\mathcal A$. By assumption there exists $u \in \mathcal D_\texttt{\#}(\mathbb T^d;U)$ such that $\mathcal B u = v$. This gives 
\[
   (2\pi \mathrm i)^{k_B} B(m) \mathfrak F u(m) = P \mathfrak F \varphi(m) \qquad \text{for all $m \in \mathbb Z^d - \{0\}$.}
\]  
Since $\varphi$ was arbitrary, this shows that $P \in \im B(m)$ for all $m \in \mathbb Z^d - \{0\}$. This proves the containment $\ker A(m) \subset \im B(m)$ for all such points. The other containment follows from similar Fourier coefficient arguments using that the sequence defines a short differential complex: $\mathcal B[\mathcal D_\texttt{\#}(\mathbb T^d;U)] \subset \ker \mathcal A \cap \mathcal D_\texttt{\#}(\mathbb T^d;V)$. 
\end{proof}

Let us give an example that shows property (1) in Theorem \ref{thm:last} is strictly weaker than the validity of the constant-rank property for dimensions $d \ge 2$. 
\begin{example}
Consider the following operator on $\R^2$, from $\R^2$ to $\R^2$ (an analogous example can be given for $d \ge 2$): 
\[
    \mathcal A(D) (u_1,u_2) = (\pi \partial_2 u_2 - \partial_1 u_1, \pi \partial_1 u_2 - \partial_2 u_2).
\]
Its associated symbol is the polynomial matrix 
\[
A(\xi) = \begin{pmatrix} \pi \xi_2 - \xi_1 & 0 \\
0 & \pi \xi_1 - \xi_2 \end{pmatrix}, \qquad \xi \in \R^2.
\]
Notice that $\rank A = 2$ on $\mathbb Z^2 - \{0\}$ and therefore $\mathcal A(D)$ is truly elliptic on $\mathcal D_\texttt{\#}(\mathbb T^d;\R^2)$, that is, its kernel restricted to $D_\texttt{\#}(\mathbb T^d;\R^2)$ is trivial.  It follows that the zero operator $\mathcal B(D) \equiv 0$ is the unique exact potential of $\mathcal A(D)$ on spaces of periodic maps. On the other hand, the rank of $A$ non-constant over $\R^2$ given that $\rank A(\pi,1) = 1$. In light of Theorem~\ref{thm:poincare}, we conclude that $\mathcal A(D)$ has no potential on full-space.
\end{example}

A direct consequence of the previous characterization and our constructions of annihilators for arbitrary fields is the following sufficiency result:
\begin{theorem}\label{thm:periodic}
Let $\mathcal A(D)$ be a constant coefficient homogeneous operator on $\R^d$ from $V$ to $W$. Further assume that
\[
    \rank A(\xi) \;\;\text{is constant on $\mathbb Z^d - \{0\}$}.
\]
Then there exists an operator $\mathcal B(D)$ on $\R^d$ from $U$ to $V$ with the following propriety: for every $v \in \mathcal C^\infty(\mathbb T^d;V)$ satisfying
\[
    \int_{\mathbb T^d} v = 0 \quad \text{and} \quad \mathcal A v = 0,
\]
there exists $u \in \mathcal C^\infty(\mathbb T^d;U)$ satisfying
\[
    \int_{\mathbb T^d} u = 0 \quad \text{and} \quad \mathcal B u = v.
\]
\end{theorem}
\begin{proof}
By homogeneity it follows that $\rank A$ is constant on $\Q^d - \{0\}$. The conclusion follows from the implication (1) $\Rightarrow$ (2) in Theorem~\ref{thm:algebra} and the previous lemma. Here, we are appealing to the observation drawn in Remark~\ref{rmk:semicontinuity of rank}.
\end{proof}

\subsection{An optimal construction} We would like to end this section with some comments on the optimality of our construction and, in particular, compare it to that of Van Schaftingen in the case where $\cA(D)$ is an elliptic operator, i.e., $A(\xi)$ is injective for every non-zero $\xi$. In the elliptic case, our $Q$ and $X$ are simply given by 
\[
    Q(\xi)(w) = a_1(\xi) \wedge \cdots \wedge a_N(\xi) \wedge w, \quad X = \bigwedge{\!\!^{N + 1}} \, W.
\]
Where, as a reminder, $N = r = \dim(V)$. The appearance of the $\binom{N}{r}$ exponent in the definition of $X$ for operators of rank $r$ originates from the fact that, while we know that some collection $a_{i_{1}}(\xi), \dots, a_{i_{r}}(\xi)$ forms a basis for $\im A(\xi)$, we do not know a priori which one (and this collection may depend on $\xi$) so we need to test all of them. In \cite[Remark 4.1]{schaft}, Van Schaftingen constructed an explicit annihilator $L$ from $W$ to $W$ given by the formula
$$
L(\xi) = \det(A(\xi)^{*}\circ A(\xi))\id_{W} - A(\xi)\circ\operatorname{adj}(A(\xi)^{*}\circ A(\xi))\circ A(\xi)^{*}\,,
$$
where $\operatorname{adj}$ is the adjugate operator.  Its associated operator $\mathcal{L}(D)$ satisfies an exactness property analogous to (2) of Theorem \ref{thm:Q}, but it has order $2\dim(V)k = 2rk$. In fact, the construction of Raiț{\u{a}} in \cite{raita2018potentials} is a generalization of this construction, 
and it also gives annihilators and potentials of order $2rk$.  While the aforementioned constructions require a higher order than ours, its target space is always $W$ thus creating a system of $\dim W$ equations. The number of equations of our construction, on the other hand, is $\binom{\dim V}{\rank A}\binom{\dim W}{\rank A+1}$. The first factor depends on  {how elliptic} $\cA(D)$ is, and the second factor on how elliptic its formal adjoint $\cA(D)^*$ is. Observe that if $\cA(D)$ is elliptic and its image has co-dimension one, then we obtain only one equation 
\[
    \cQ(D) = \det (a_1(D),\dots,a_N(D),u).
\]
Regarding the existence of an \emph{optimal} annihilator for elliptic operators, Van Schaftingen gave (see~\cite[Lemma 4.4]{schaft}) an abstract construction of an homogeneous operator $\mathcal J(D)$, which is minimal in the following sense
\[
K(D) \circ A(D) = 0 \quad \Longrightarrow \quad \mathcal K(D) = \mathcal P(D) \circ \mathcal J(D)
\]
for some linear operator $\mathcal P(D)$ and 
\begin{equation}\label{eq:optimal}
    \ker J(\xi) = \im A(\xi) \quad \text{for all $\xi \in \mathbb K^d - \{0\}$.}
\end{equation}
Note that our operator $\mathcal{Q}(D)$, while improving on the order of those constructed in \cite{raita2018potentials, schaft} may not be of minimal order as it can be seen by comparing our construction with De Rham's sequence for dimensions $d\ge 3$. 

The following result is an extension of Van Schaftingen's optimal construction for elliptic operators.


\begin{proposition}\label{prop:construction}
Let $\cA(D)$ be a homogeneous degree $k$ operator on $\mathbb R^d$ from $V$ to $W$. Further, assume that 
\[
\rank A(\xi) = r \quad \text{on $\R^d - \{0\}$}.
\]

There exists a finite-dimensional space $G$ and a homogeneous linear differential operator $\mathcal J(D)$ on $\R^d$, with symbol
\[
  \xymatrix@1{{W}\ar[r]^-{J(\xi)} \; & \;
	G}
\]
satisfying the following properties:
\begin{enumerate}[(i)]
    \item $\mathcal J(D)$ is an exact annihilator of $\cA(D)$, that is, 
    \[
    \ker J(\xi) = \im A(\xi) \quad \text{for all nonzero $\xi \in \mathbb{R}^d$}.
    \]
    \item 
    If $\mathcal K(D)$ is an annihilator from $W$ to $X$, then there exists $\mathcal P(D)$ from $G$ to $X$ such that 
    \[
        \mathcal K(D) = \mathcal P(D) \circ \mathcal J(D).
    \]
\item In particular, 
    \[
        \begin{cases}
            \ker \mathcal J(D) \subset \ker \mathcal K(D) \\
        \;\,\deg J(\xi)  \le \deg K(\xi)
        \end{cases} \quad \text{for all $\mathcal K(D) \in \mathcal K$.}
    \]
     If moreover $\mathcal K(D)$ is an exact annihilator of $\cA(D)$ of minimal order, then also
    \[
        \dim G \le \dim X.
    \]
\end{enumerate}
\end{proposition}
\begin{proof} If $A: V \to W$ is a map of vector spaces, we denote its adjoint by $A^{\star}: W \to V$, defined by the property $(Av, w)_{W} = (v, A^{\star}w)_{V}$, where the latter pairings are fixed inner products of $V,W$.

Let us first assume that $r = 2m+1$ is odd. In Theorem \ref{thm:algebra} we have constructed a potential $\cB(D)$ of order $rk$ to $\cA(D)$. On the other hand, we have the operator $\cA(D)\cA^{\star}(D): V \to V$ of order $2k$. We consider the operator
\begin{align*}
\mathcal{H}(D) & \phantom{:}=  \mathcal H_1(D) \oplus \mathcal H_2(D) \\
&\coloneqq \cB^{\star}(D) \oplus \cA(D)(\cA^{\star}(D)\cA(D))^{m}: V \to U \oplus W
\end{align*}
We claim that $\mathcal{H}$ is elliptic. Indeed, for a nonzero vector $\xi \in \R^d$ and every $m \ge 0$, $\ker A(\xi)  = \ker(A(\xi)[A^{\star}(\xi)A(\xi)]^{m})$, so it is enough to check that $\ker (B^{\star}(\xi))\cap\ker(A(\xi)) = \{0\}$. This follows because $B$ is a potential of $A$, and $\ker B^{\star}(\xi) = (\im B(\xi)^{\perp})^*$.
Note also that $\mathcal{H}$ is homogeneous of degree $rk$. In particular, we may apply Van Schaftingen's construction to find a homogeneous exact annihilator  $\widetilde{\mathcal{J}}(D) = \mathcal{J}_1(D) \oplus \mathcal{J}_2(D)$ of $\mathcal{H}(D)$, satisfying (ii) for $\mathcal H(D)$ instead of $\cA(D)$. 

We claim that $\mathcal{J}(D)\coloneqq \mathcal{J}_{2}(D)$ satisfies (i)-(iii). 

First, we show that (i) holds. We fix $\xi \in \R^d$ a nonzero vector. By construction, we have that $ \ker J(\xi) = \ker \tilde J(\xi) \cap W$. On the other hand, by the exactness of $\tilde J(\xi)$, we deduce that
\begin{align*}
    h \in \ker \tilde J(\xi) \cap W \quad & \Leftrightarrow \quad  h \in \im H(\xi) \cap W \\
    \quad & \Leftrightarrow 
    \quad h = A(\xi)(A(\xi)^\star A(\xi))^m v, \quad v \in \ker B(\xi)^\star.
\end{align*}
This shows that $\ker J(\xi) = A(\xi)(A(\xi)^\star A(\xi))^m[\ker B(\xi)^\star]$. Since $A(\xi)$ is an exact annihilator of $B(\xi)$, it further holds 
$V = \ker B(\xi)^\star \oplus \ker A(\xi)$.
We thus conclude that
\[
    \ker J(\xi) = A(\xi)(A(\xi)^\star A(\xi))^m[V] = A(\xi)[V] = \im A(\xi).
\]
This proves (i).

Next, we prove (ii). To this end, let $\mathcal K(D)$ be an annihilator of $\cA(D)$ and consider the operator $\tilde {\mathcal K}(D) \coloneqq \mathcal K(D) \circ \pi_W$, where $\pi_W : U \oplus V \to W$ is the canonical projection onto the $W$-coordinate. Clearly, $\tilde {\mathcal K}(D)$ is an annihilator of $\mathcal H(D)$ and hence, by construction, it factors through $\tilde {\mathcal J}(D)$. We may thus find  $\mathcal P(D)$ such that $\tilde {\mathcal K}(D) = \mathcal P(D) \circ \tilde {\mathcal J}(D)$. By construction, we get $\mathcal K(D)  = \mathcal P(D) \circ {\mathcal J}(D)$. This shows that every annihilator of $\cA(D)$ factors through $\mathcal J(D)$ which is precisely the statement in (ii). As a side note, notice that it necessarily holds
\[
    G = \sum_{\xi \in \R^d} \im J(\xi),
\]  
for otherwise the factorization property would fail. 

Lastly, we claim that (iii) follows from (i)-(ii). The fist part of (iii) follows directly from (ii).  For the the second part, we notice that if $\mathcal K(D)$ is an exact annihilator of minimal order, then, by minimality, the property (ii) guarantees the existence of a linear map $L : G \to X$ such that $\mathcal K(D) = L \circ \mathcal J(D)$. Moreover, since $\ker J(\xi) = \ker K(\xi)$ for all nonzero $\xi$, we deduce that $L$ is one-to-one when restricted to $G = \sum_{\xi \in \R^d} \im J(\xi)$. In particular $\dim G = \dim L[G] \le \dim X$, which is the sought statement. This completes the proof of (i)-(iii) when $r = 2m + 1$ for some $m \ge 0$.

The case when $r = 2m$ is even follows analogously by considering the exact potential $\mathcal{Q}(D): U \to W$ of $\mathcal{A}^{\star}(D): W \to V$ and set
\[
\mathcal{H}(D) := \mathcal{Q}^{\star}(D) \oplus (\mathcal{A}(D)\mathcal{A}^{\star}(D))^{m}: W \to U \oplus W
\]
\noindent we claim that $\mathcal{H}(D)$ is elliptic. Indeed, for any nonzero $\xi \in \R^d$ it holds $\ker(A(\xi)A^{\star}(\xi)) = \ker A^{\star}(\xi)$. It is therefore enough to check that $\ker Q^{\star}(\xi) \cap\ker A^{\star}(\xi)$ is trivial, and this follows since $\ker Q^{\star}(\xi) = (\ker A^{\star}(\xi)^{\perp})^*$. Now we run a similar reasoning as in the odd $r$ case. This completes the proof.
\end{proof}

\section{Examples}

Below we show how our construction can be used to find an optimal annihilator of two well-known and relevant operators. 

\subsection{Gradients}  The gradient operator $D$ acts on vector fields $u: \R^d \to \R^m$ as
\[
    Du = \left(\frac{\partial u_i}{\partial x_j}\right)_{i,j}, \qquad 1\le i \le m, \; 1 \le j \le d.
\]
An application of the Fourier transform shows that the columns of its symbol are precisely
\[
e_{j} \otimes \xi \coloneqq e_j \xi^t, \qquad \xi \in \R^d, j \in \{1,\dots,m\}.
\]
Note that we have
\[
(e_{i} \otimes \xi) \wedge \cdots \wedge (e_m \otimes \xi) = \sum_{i_{1}, \dots, i_{m} = 1}^{d} P_{1, i_{1}} \wedge \cdots \wedge P_{m, i_{m}}\xi_{i_{1}}\cdots \xi_{i_{m}},
\]
where $P_{i,j}$ is the $m \times n$ matrix defined by $(P_{i,j})_{a,b} = \delta_{i,a}\delta_{j,b}$. It follows that the annihilator $Q(\xi)$ is
\[
Q(\xi)w = \sum_{\substack{i_1, \dots, i_m,q = 1,\dots,d\\ p = 1,\dots,m}} (P_{1i_1} \wedge \cdots \wedge P_{m,i_m} \wedge P_{p,q}) \xi_{i_1}\cdots\xi_{i_m}w_{p,q}
\]
Let $k \in I = \{1,\dots,m\}$ and let $1 \le i_k^{(1)} < i_k^{(2)} \le d$. 
From the expression above it follows that the coefficient of $Q(\xi)w$ corresponding to the basis element $P_{1, i_{1}}\wedge \cdots \wedge P_{k, i_{k}^{(1)}}\wedge P_{k, i_{k}^{(2)}}\wedge \cdots \wedge P_{m, i_{m}}$ is given by 
\[
     (-1)^{d-k-1} \prod_{j \in I - \{k\}} \xi_{i_j} \left(\xi_{i_k^{(1)}} w_{k,i_k^{(2)}} - \xi_{i_k^{(2)}} w_{k,i_k^{(1)}}\right).
\]

Thus, up to an isomorphism, we have:
\[
\mathcal Q(D) = D^{m-1}\Curl,
\]
where ``$\curl$'' is the row-wise curl operator on $\R^{m \times d}$-valued fields, that is,
\[
    \curl M = \begin{pmatrix}
    \operatorname{curl} M^1 \\
    \vdots \\
    \operatorname{curl} M^m
    \end{pmatrix}, \qquad \operatorname{curl} (v_1,\dots,v_d) = (\partial_{r} v_s - \partial_s v_r)_{r,s = 1,\dots,d}.
\]
Since $D^{m-1}$ is elliptic, we observe that a minimal annihilator of the gradient is the row-wise curl, as one would expect from de Rham's sequence.

\subsection{The equations in linear elasticity}Consider the symmetric gradient in three dimensions, given by
\[
    Eu = \sym (Du) = \frac 12 (Du + Du^t), \qquad u : \R^3 \to \R^3.
\]
This is an operator from $\R^3 \to \R^{3 \times 3}_\sym$, the space of symmetric $3 \times 3$-matrices, that has a basis $\{S_{i,j}\}_{i \leq j}$, where $S_{i,j}$ is the symmetrization of $P_{i,j}$. We will also consider $S_{i,j}$ for $j \leq i$, with the understanding that $S_{i,j} = S_{j,i}$. 
Notice that the columns of the symbol map $L(\xi)$ are 
\[
    l(\xi)_i = \sum_{j} (1 + \delta_{ij})\xi_j S_{ij}, \qquad i = 1,2,3.
\]
Let $w \in \R^{3 \times 3}_\sym$ and consider
\begin{equation}\label{eq:Qsymg}
    Q(\xi)w = l(\xi)_1 \wedge l(\xi)_2 \wedge l(\xi)_3 \wedge w. 
\end{equation}
This is an element of $\bigwedge^{4}\R^{3 \times 3}_\sym$, that is a $\binom{6}{4} = 15$-dimensional space with basis $S_{i_{1}, j_{1}}\wedge S_{i_{2}, j_{2}}\wedge S_{i_{3}, j_{3}}\wedge S_{i_{4}, j_{4}}$, where $(i_{m}, j_{m}) \in \{1,2,3\} \times \{1,2,3\}$ for $m = 1, \dots, 4$ and the obvious restrictions on $(i_{m}, j_{m})$. We will find the coefficient of each basis vector in the expression  \eqref{eq:Qsymg}.

Let us start with those basis vectors that, up to a reordering of the wedge factors, are of the form $S_{11} \wedge S_{22} \wedge S_{33} \wedge S_{ij}$ with $i \neq j$. There are $3$ of them, and a quick computation shows that the $S_{ii}\wedge S_{jj} \wedge S_{kk} \wedge S_{ij}$ coefficient of $Q(\xi)w$ is given (up to a multiplicative constant) by
\begin{equation}\label{eq:symg1}
     \xi_k[2\xi_i\xi_j w_{ij} -  \xi_i^2 w_{jj} - \xi_j^2 w_{ii}].
\end{equation}

Moving on, we now look at the coefficient of an element of the form $S_{ii} \wedge S_{jj} \wedge S_{ij} \wedge S_{jk}$, for $i \neq j$ and $k \not\in \{i, j\}$. There are $\binom{3}{2}\times 2 = 6$ basis elements of this form, and it is straightforward to see that the $S_{ii}\wedge S_{jj}\wedge S_{ij} \wedge S_{jk}$ coefficient is given (up to a sign) by
\begin{equation}\label{eq:symg2}
\xi_j [2\xi_i\xi_j w_{ij} - \xi_i^2 w_{jj} - \xi_j^2 w_{ii}].
\end{equation}
In total, Equations \eqref{eq:symg1} and \eqref{eq:symg2} give us 9 equations of order 3, which can be presented as $D \circ \mathcal F(D)$ and since the gradient is elliptic, they may be reduced to $3$ equations of order 2:
\begin{equation}\label{eq:symg3}
2\xi_{i}\xi_{j}w_{ij} - \xi_{i}^{2}w_{jj} - \xi_{j}^{2}w_{ii} \qquad (i \neq j).
\end{equation}

We still need to find the coefficients of $15 - 9 = 6$ basis vectors. There are $\binom{3}{2} = 3$ basis vectors of the form $S_{ii}\wedge S_{jj} \wedge S_{ik} \wedge S_{jk}$ and the coefficient of this basis vector is, up to a multiplicative constant, given by
\begin{equation}\label{eq:symg4}
    \xi_i (2\xi_i\xi_j w_{jk} - \xi_i \xi_k w_{jj}) - \xi_j (2\xi_i\xi_j w_{ik} - \xi_j\xi_k w_{ii}).
\end{equation}

The coefficients of the three remaining basis vectors $S_{ii} \wedge S_{ij} \wedge S_{ik} \wedge S_{jk}$ are given (up to multiplicative constant) by
\begin{equation}\label{eq:symg5}
\xi_i [ \xi_j\xi_k   w_{ii} - \xi_i \xi_k w_{ij}    - \xi_i \xi_j w_{ik} +  \xi_i^2 w_{jk} ]
\end{equation}

We would like to observe that, just as in~\eqref{eq:symg2},  these equations can be simplified to a system of equations of order $2$. If $\xi_1\xi_2\xi_3 \neq 0$, then all right factors of \eqref{eq:symg5} have to vanish. If, on the other hand, $\xi_i = 0$ for some $i$ then we need to show that for $\{j, k\} = \{1, 2, 3\} - \{i\}$, $\xi_{j}\xi_{k}w_{ii} = 0$. But thanks to \eqref{eq:symg4} we have that $\xi_{j}\xi_{k}w_{ii} = 0$. In particular~\eqref{eq:symg5} vanishes if and only if so does
\begin{equation}\label{eq:symg6}
     \xi_j\xi_k   w_{ii} - \xi_i \xi_k w_{ij}    - \xi_i \xi_j w_{ik} +  \xi_i^2 w_{jk}.  
\end{equation}
Conversely, it is straightforward to verify that~\eqref{eq:symg4}-\eqref{eq:symg5} vanish provided that~\eqref{eq:symg6} vanishes. It thus follows from our construction that the operator
\[
    \cQ'(D) = (\partial_{jk}   w_{ii} - \partial_{ik} w_{ij}    - \partial_{ij} w_{ik} +  \partial_{ii} w_{jk})_{i,j,k}, \qquad (i \notin\{j,k\}, \; 1 \le j \le k \le 3),
\]
is an exact annihilator of $E(\xi)$.  

Finally, we would like to observe that $\cQ'(D)$ is equivalent to the St.-Venant compatibility equations
\[
   \sum_{i = 1}^3 (\partial_{ji}w_{ik} + \partial_{ki}w_{ij} - \partial_{ij}w_{ii} - \partial_{ii}w_{jk} )_{j,k}, \qquad j,k = 1,2,3,
\]
which is well-known to be an optimal exact annihilator of $E(D)$. 
Note that the equations \eqref{eq:symg3} are precisely the diagonal equations of the St.-Venant system, and the equations \eqref{eq:symg5} are the off-diagonal equations of the St.-Venant system.

\subsection*{Acknowledgements} We thank Andr\'e Guerra for pointing out that the homology over the space of Schwartz functions may be nontrivial. AAR was supported by the Fonds de la Recherche Scientifique - FNRS under Grant No 40005112 (Compensated compactness and fine
properties of PDE-constrained measures). JS  is grateful for the financial support and hospitality of the Max Planck Institute for Mathematics, where this work was carried out.

\bibliography{bibliography.bib}

\end{document}